\documentclass[12pt]{amsart} 

\usepackage[english]{babel}
\usepackage[latin1]{inputenc}
\usepackage{amsmath}
\usepackage{amsthm}
\usepackage{amssymb}
\usepackage{tikz}
\usepackage[all]{xy}
\usepackage{imakeidx}
\usepackage{appendix}
\usepackage{tikz-cd}
\usepackage{verbatim}
\usepackage{enumitem}
\usepackage{multicol}

\usepackage{cleveref}

\usepackage{mathtools}

\newtheorem{thm}{Theorem}[section]

\newtheorem{prop}[thm]{Proposition}
\newtheorem{lem}[thm]{Lemma}
\newtheorem{lemma}[thm]{Lemma}
\newtheorem{cor}[thm]{Corollary}
\newtheorem{conj}[thm]{Conjecture}

\theoremstyle{definition}

\theoremstyle{remark}
\newtheorem{rmk}[thm]{Remark}
\newtheorem{reduction}[thm]{Reduction}
\numberwithin{equation}{section}

\newcommand{\Z}{\mathbb Z}
\newcommand{\G}{\mathbb G}
\newcommand{\GL}{\operatorname{GL}}
\newcommand{\bbG}{\mathbb G}
\newcommand{\bbZ}{\mathbb Z}
\newcommand{\bbP}{\mathbb P}
\renewcommand{\P}{\mathbb P}

\newcommand{\trdeg}{\operatorname{trdeg}}
\renewcommand{\c}{\subseteq}
\newcommand{\A}{\mathbb A}

\newcommand{\mc}[1]{\mathcal{#1}}
\newcommand{\cl}{\overline}
\newcommand{\set}[1]{\{#1\}}
\renewcommand{\phi}{\varphi}
\newcommand{\ed}{\operatorname{ed}}
\newcommand{\on}[1]{\operatorname{#1}}

\newcommand{\ang}[1]{\left \langle{#1}\right \rangle}

\author{Zinovy Reichstein}
\address{Department of Mathematics\\
 University of British Columbia\\
 Vancouver, BC V6T 1Z2\\Canada}
 \email{reichst@math.ubc.ca}
\thanks{Zinovy Reichstein was partially supported by
 National Sciences and Engineering Research Council of
 Canada Discovery grant 253424-2017.}

\author{Federico Scavia}
\email{scavia@math.ubc.ca}
\thanks{Federico Scavia was partially supported by a graduate fellowship from the University of British Columbia.}

	

\title[Essential dimension]{On the essential dimension of an algebraic group whose connected component is a torus}
\begin{document}

	\begin{abstract}
	Let $p$ be a prime integer, $k$ be a $p$-closed field of characteristic $\neq p$, 
	$T$ be a torus defined over $k$, $F$ be a finite $p$-group, and $1 \to T \to G \to F \to 1$ be an exact sequence of algebraic groups. 
	In this paper we study the essential dimension $\ed(G; p)$ of $G$ at $p$. 
	R.~L\"otscher, M.~MacDonald, A.~Meyer, and the first author showed that 
	\[ \min \, \dim(V) - \dim(G) \leqslant \ed(G; p) \leqslant \min \, \dim(W) - \dim(G), \] 
	where $V$ and $W$ range over the $p$-faithful and $p$-generically free $k$-representations of $G$, respectively. 
	This generalizes the formulas for the essential dimension at $p$ of a finite $p$-group due to N.~Karpenko and A.~Merkurjev 
	(here $T = \{ 1 \}$) and of a torus, due to L\"otscher et al.
	(here $F = \{ 1 \}$). In both of these cases every $p$-generically free representation of $G$
	is $p$-faithful, so the upper and lower bounds on $\ed(G; p)$ given above coincide. In general there is a gap between these bounds.
	L\"otscher et~al.~conjectured that the upper bound is, in fact, sharp; that is, $\ed(G; p) = \min \, \dim(W) - \dim(G)$, 
	where $W$ ranges over the $p$-generically free representations, as above. We prove this conjecture in the case, where $F$ is diagonalizable. Moreover, we give an explicit way to compute $\min \, \dim(W)$ in this case. As an application of our main theorem we compute $\ed(G; p)$, where $G$ is the normalizer of a split maximal torus in a split simple algebraic group, in all previously inaccessible cases. 
	\end{abstract}
	
	\maketitle
	
\section{Introduction}

Let $p$ be a prime integer and $k$ be a $p$-closed field of characteristic $\neq p$. 
That is, the degree of every finite extension $l/k$ is a power of $p$.
Consider an algebraic group $G$ defined over $k$, which fits into the exact sequence
\begin{equation} \label{e.torus}
\xymatrix{ 1 \ar@{->}[r] &  T \ar@{->}[r] & G \ar@{->}[r]^{\pi} & F \ar@{->}[r] & 1,} 
\end{equation}
where $T$ is a (not necessarily split) torus and $F$ is a (not necessarily constant)
finite $p$-group defined over $k$. We say that a representation $G \to \GL(V)$
is $p$-faithful if its kernel is a finite subgroup of $G$ of order prime to $p$ and $p$-generically free if the isotropy subgroup $G_v$ is a finite group of order prime to $p$ for $v \in V(\overline{k})$ in general position. We denote by $\eta(G)$ (respectively, $\rho(G)$) the smallest dimension of a $p$-faithful (respectively, $p$-generically free) representation. R.~L\"otscher, M.~MacDonald, A.~Meyer, and the first 
author~\cite[Theorem 1.1]{lotscher2013essential2} showed 
that the essential $p$-dimension $\ed(G; p)$ satisfies the inequalities
\begin{equation} \label{e.lmmr}
\eta(G)- \dim(G) \leqslant \ed_k(G; p) \leqslant \rho(G) - \dim(G).
\end{equation}

The inequalities~\eqref{e.lmmr} represent a common generalization of the formulas for the essential $p$-dimension of
a finite constant $p$-group, due to N.~Karpenko and A.~Merkurjev~\cite[Theorem 4.1]{km2}
(where $T = \{ 1 \}$), and of an algebraic torus, due to L\"otscher et al.~\cite{lotscher2013essential} (where $F = \{ 1 \}$).
In both of these cases, every $p$-faithful representation of $G$ is $p$-generically free, and thus $\eta(G)=\rho(G)$. In general, 
$\eta(G)$ can be strictly smaller than $\rho(G)$. L\"otscher et~al.~conjectured that
the upper bound of~\eqref{e.lmmr} is, in fact, sharp. 


\begin{conj} \label{conj.lmmr} 
Let $p$ be a prime integer, $k$ be a $p$-closed field of characteristic $\neq p$,
and $G$ be an affine algebraic group defined over $k$. Assume that the connected component
$G^0 = T$ is a $k$-torus, and the component group $G/G^0 = F$ is a finite $p$-group.
Then \[\ed(G; p) = \rho(G)-\dim G,\] where $\rho(G)$ is the minimal 
dimension of a $p$-generically free $k$-representation of $G$.
\end{conj}

Informally speaking, the lower bound of~\eqref{e.lmmr} is the strongest lower bound on $\ed(G; p)$ one can hope
to prove by the methods of~\cite{km2}, \cite{lotscher2013essential}, and \cite{lotscher2013essential2}. In the case, where
the upper and lower bounds of~\eqref{e.lmmr} diverge, Conjecture~\ref{conj.lmmr} calls for a new approach.

Conjecture~\ref{conj.lmmr} appeared in print in~\cite[Section 7.9]{reichstein2010essential} on the list of open problems in the theory of essential dimension.
The only bit of progress since then has been a proof in the special case, where $G$ is a semi-direct product of a cyclic group
$F = \mathbb Z/ p \mathbb Z$  
of order $p$, and a split torus $T = \bbG_{\on{m}}^n$, due to M.~Huruguen~\cite{huruguen}.  
Huruguen's argument relies on the classification of integral representations of $\mathbb Z / p \mathbb Z$ 
due to F.~Diederichsen and I.~Reiner~\cite[Theorem 74.3]{curtis-reiner}. So far this approach has resisted all attempts to
generalize it beyond the case, where $G \simeq \bbG_{\on{m}}^n \rtimes (\mathbb Z/ p \mathbb Z)$.

Note that $\eta(G)$ is often accessible by cohomological and/or combinatorial techniques; see Section~\ref{sect.index} and
Lemma~\ref{symmrank}, as well as the remarks after this lemma. 
Computing $\rho(G)$ is usually a more challenging problem.
The purpose of this paper is to establish \Cref{conj.lmmr} in the case, where $F$ is a diagonalizable abelian $p$-group.
Moreover, our main result also gives a way of computing $\rho(G)$ in this case. 

\begin{thm} \label{thm.main} Let $p$ be a prime integer, $k$ be a $p$-closed field of characteristic $\neq p$,
and $G$ be an extension of a (not necessarily constant) diagonalizable $p$-group $F$ by a (not necessarily split) torus $T$, as in~\eqref{e.torus}. Then

\smallskip
(a) $\ed(G; p) = \rho(G)-\dim G$.

\smallskip
(b) Moreover, suppose $V$ is a $p$-faithful representation of $G$ of minimal dimension, $k$ is the algebraic closure of $k$, and $S \subset G_{\overline{k}}$ is a stabilizer in general position
for the $G_{\overline k}$-action on $V_{\overline k}$. Then $\rho(G) = \eta(G) + \on{rank}_p(S)$. 
\end{thm}

Here $\on{rank}_p(S)$ is the largest $r$ such that $S$ contains a subgroup isomorphic to $\mu_p^r$.
Most of the remainder of this paper (Sections~\ref{sect.stabilizers}-\ref{sect.proof})
will be devoted to proving Theorem~\ref{thm.main}.
A key ingredient in the proof is the Resolution Theorem~\ref{divisorial}, which is based, 
in turn, on an old valuation-theoretic result of M.~Artin and 
O.~Zariski~\cite[Theorem 5.2]{artin1986neron}. In Section~\ref{sect.normalizers} we will use Theorem~\ref{thm.main}
to complete the computation of $\ed(N; p)$ initiated in~\cite{mr} and~\cite{macdonald}.  
Here $N$ is the normalizer of a split maximal torus in a split simple algebraic group.
%

\section{Stabilizers in general position}
\label{sect.stabilizers}

In this section we will assume that the base field $k$ is algebraically closed.
Let $G$ be a linear algebraic group defined over $k$. A $G$-variety $X$ is called primitive if $G$ transitively 
permutes the irreducible components of $X$.

Let $X$ be a primitive $G$-variety. A subgroup $S \subset G$ is called a stabilizer in general position for the $G$-action on $X$
if there exists an open $G$-invariant subset $U \subset X$ such that $\on{Stab}_G(x)$ is conjugate to $S$ for every $x \in U({k})$.
Note that a stabilizer in general position does not always exist. When it exists, it is unique up to conjugacy. 

\begin{lem} \label{lem.sgp-existence} 
Let $G$ be a linear algebraic group over $k$ and $X$ be a primitive quasi-projective $G$-variety. 
Assume that the connected component $T = G^0$ is a torus and the component group $F = G/G^0$ is finite of order prime to $\on{char}(k)$.
Then there exists a stabilizer in general position $S \subset G$.
\end{lem}

\begin{proof} 
After replacing $G$ by $\overline{G} := G/(K \cap T)$, where $K$ is the kernel of the $G$-action on $X$, we may assume that
the $T$-action on $X$ is faithful and hence, generically free. In other words, for $x \in X(k)$ in general position, 
$\on{Stab}_G(x) \cap T = 1$; in particular, $\on{Stab}_G(x)$ is a finite $p$-group. Since $\on{char}(k) \neq p$, Maschke's 
theorem tells us that $\on{Stab}_G(x)$ is linearly reductive. Hence, for $x \in X(k)$ in general position,
$\on{Stab}_G(x)$ is $G$-completely reducible; see~\cite[Lemma 11.24]{jantzen}. 
The lemma now follows from~\cite[Corollary 1.5]{martin}.
\end{proof}

\begin{rmk} The condition that $X$ is quasi-projective can be dropped if $k = \mathbb C$; see \cite[Theorem 9.3.1]{richardson}.
With a bit more effort this condition can also be removed for any algebraically closed base field $k$ of characteristic $\neq p$. 
Since we shall not need this more general variant of~\Cref{lem.sgp-existence}, we leave its proof as an exercise for the reader.
\end{rmk}

We define the (geometric) $p$-rank $\on{rank}_p(G)$ of an algebraic group $G$ to be the largest integer $r$ such that $G$ contains a subgroup isomorphic to $\mu_p^r = \mu_p \times \dots \times \mu_p$ ($r$ times).

\begin{lemma} \label{lem.p-rank} 
Let $X$ be a normal $G$-variety and $Y \subset X$ be a $G$-invariant prime divisor of $X$. 
Let $S_X$ and $S_Y$ be stabilizers in general position of the $G$-actions on $X$ and $Y$, 
respectively. Assume that $p$ is a prime and $\on{char}(k) \neq p$. Then: 

\smallskip
(a) $\on{rank}_p(S_Y) \leqslant \on{rank}_p(S_X) + 1$.

\smallskip
(b) Assume the $G$-action on $X$ is $p$-faithful. Denote the kernel of the $G$-action on $Y$ by $N$. Then there is a group homomorphism
$\alpha \colon N \to \bbG_{\on{m}}$ such that $\on{Ker}(\alpha)$
does not contain a subgroup of order $p$.
\end{lemma}

\begin{proof} 
Let $U \subset X$ be a $G$-invariant dense open subset of $X$ such that $\on{Stab}_G(x)$ is conjugate to $S$ for every $x \in U(k)$.
If $Y \cap U \neq \emptyset$, then $S_Y = S_X$, and we are done. Thus we may assume that $Y$ is contained in $Z = X \setminus U$.
Since $Y$ is a prime divisor in $X$, it is an irreducible component of $Z$. After removing all other irreducible components of $Z$ from $X$, we may assume that
$Z = Y$. Since $X$ is normal, $Y$ intersects the smooth locus of $X$ non-trivially. Choose a $k$-point $y \in Y$ such that both $X$ and $Y$ are smooth at $y$
and $\on{Stab}_G(y)$ is conjugate to $S_Y$.
After replacing $S_Y$ by a conjugate, we may assume that $\on{Stab}_{G}(y) = S_Y$. The group $\on{Stab}_G(y)$ acts on the tangent spaces
$T_y(X)$ and $T_y(Y)$, hence on the $1$-dimensional normal space $T_y(X)/T_y(Y)$. This gives rise to a character 
$\alpha \colon S_Y \to \bbG_{\on{m}}$.

(a) Assume the contrary: $S_Y$ contains $\mu_p^{r+2}$, where $r = \on{rank}_p(S_X)$. 
Then the kernel of $\alpha$ contains a subgroup $\mu \simeq \mu_p^{r + 1}$. By Maschke's Theorem, the natural projection $T_y(X) \to T_y(X)/T_y(Y)$ is $\mu$-equivariantly split. Equivalently, there exists a $\mu$-invariant tangent vector $v \in T_y(X)$ which does not belong to $T_y(Y)$. By the Luna Slice Theorem, 
\begin{equation} \label{e.luna} T_y(X)^{\mu}=T_y(X^{\mu}). \end{equation}
For a proof in characteristic $0$, see~\cite[Section 6.5]{popov1994}. Generally speaking, Luna's theorem fails in prime characteristic, 
but~\eqref{e.luna} remains valid, because $\mu$ is linearly reductive; see~\cite[Lemma 8.3]{bardsley-richardson}.
Now observe that since $\mu$ does not fit into any conjugate of $S_X$, 
the subvariety $X^{\mu}$ 
is contained in $Y = X \setminus U$. Thus $v \in T_y(X)^{\mu} = T_y(X^{\mu}) \subset T_y(Y)$, a contradiction.

\smallskip
(b) Assume the contrary: $\on{Ker}(\alpha)$ contains a subgroup $H$ of order $p$. 
Then $H$ (i) fixes a smooth point $y$ of $X$ and 
(ii) acts trivially on both $T_y(Y)$ and $T_y(X)/T_y(Y)$ and hence (since $H$ is linearly reductive) on
$T_y(X)$. It is well known that (i) and (ii) imply that $H$ acts trivially on $X$; see, e.g., the proof of \cite[Lemma 4.1]{gille-reichstein}. 
This contradicts our assumption that the $G$-action on $X$ is $p$-faithful.
\end{proof}
    
\section{Covers}
\label{sect.covers}

Let $k$ be an arbitrary field, and let $G$ be a linear algebraic group defined over $k$. As usual, we will denote the algebraic closure of $k$ by $\overline{k}$.
A $G$-variety $X$ is called primitive if the $G_{\cl{k}}$-variety $X_{\cl{k}}$ is primitive. A dominant $G$-equivariant rational map $X \dasharrow Y$ of primitive $G$-varieties 
is called a cover of degree $d$ if $[k(X): k(Y)] = d$. Here if $X_1, \dots, X_n$ are the irreducible components of $X$, then 
$k(X)$ is defined as $k(X_1) \oplus \dots \oplus k(X_n)$.

\begin{lem} \label{lem.covers} Let $p$ be a prime integer, $G$ be a smooth algebraic group 
such that $G/G^0$ is a finite $p$-group,
$W$ be an irreducible $G$-variety, $Z \subset W$ be an irreducible divisor in $W$,
and $\tau \colon X \dasharrow W$ be a $G$-equivariant cover of degree prime to $p$. 
Then there exists a commutative diagram of $G$-equivariant maps
\[ \xymatrix{  
                                        D \ar@{->}[dd]_{\tau'} \ar@{^{(}->}[r] & X' \ar@{-->}[d]^{\alpha}  \ar@/^2pc/[dd]^n  \cr
                                                             & X \ar@{-->}[d]^{\tau}  \cr
Z \ar@{^{(}->}[r] & W } \]
such that $X'$ is normal, $\alpha$ is a birational isomorphism, 
$D$ is an irreducible divisor in $X'$, and $\tau'$ is a cover of $Z$ of degree prime to $p$.
\end{lem}

\begin{proof} Let $X'$ be the normalization of $W$ in the function field $k(X)$. Since $G$ acts on $W$ and $X$ compatibly, 
there is a $G$-action on $X'$ such that the normalization map $n \colon X' \to W$ is $G$-equivariant. Over the dense open subset
of $W$ where $\tau$ is finite, $n$ factors through $X$. Thus $n$ factors into a composition of a birational isomorphism 
$\alpha \colon X' \dasharrow X$ and $\tau \colon X \dasharrow W$. This gives us the right column in the diagram.

To construct $D$, we argue as in the proof of~\cite[Proposition A.4]{reichstein-youssin}.
Denote the irreducible components of the preimage of $Z$ under $n$ by $D_1, \dots, D_r \subset X'$.
These components are permuted by $G$. Denote the orbits of this permutation action by $\mathcal{O}_1, \ldots, \mathcal{O}_m$. 
After renumbering $D_1, \dots, D_n$, we may assume that $D_i \in \mathcal{O}_i$ for $i = 1, \dots, m$. 
By the ramification formula (see, e.g., \cite[Corollary 6.3, p. 490]{lang2002algebra}), 
\[ d = \sum_{i = 1}^m \, | \mathcal{O}_i| \cdot [D_i : Z] \cdot e_i , \]
where $[D_i: Z]$ denotes the degree of the cover $n_{| D_i} \colon D_i \to Z$, and
$e_i$ is the ramification index of $n$ at the generic point of $D_i$. Since $d$ is prime to $p$, and
each $| \mathcal{O}_i |$ is a power of $p$, we conclude that there exists an $i \in \{ 1, \dots, m \}$ such that
$|\mathcal{O}_i| = 1$ (i.e., $D_i$ is $G$-invariant) and $[D_i : Z]$ is prime to $p$. We now set $D = D_i$ and
$\tau' = n_{| D_i}$.
\end{proof}

\begin{lem} \label{lem.sgp-in-covers} Let $G$ be a linear algebraic group over an algebraically closed field $k$, $p \neq \on{char}(k)$ be a prime number and
$\tau \colon X \dasharrow W$ be a cover of $G$-varieties of degree $d$. Assume stabilizers in general position for the $G$-actions on $X$ and $W$ exist; denote them by
$S_X$ and $S_W$ respectively. Assume $d$ is prime to $p$. 

\smallskip
(a) If $H$ is a finite $p$-subgroup of $S_W$, then $S_X$ contains a conjugate of $H$.

\smallskip
(b) $\on{rank}_p(S_X) = \on{rank}_p(S_W)$.
\end{lem}

\begin{proof} (a) After replacing $W$ by a dense open subvariety, we may assume that the stabilizer of every point in $W$ is a conjugate of $S_W$.
Furthermore, after replacing $X$ by the normal closure of $W$ in $k(X)$, we may assume that $\tau$ is a finite morphism. 
We claim that $W^{S_W} \subset \tau(X^H)$. Indeed, suppose $w \in W^{S_W}$. Then $H$ acts on $\tau^{-1}(w)$, which is a zero cycle on $X$ of degree $d$. Since $H$ is a $p$-group,
it fixes a $k$-point in $\tau^{-1}(w)$. Hence, $X^H \cap \tau^{-1}(w) \neq \emptyset$ or equivalently, $w \in \tau(X^H)$. This proves the claim.

Since the stabilizer of every point of $W$ is conjugate to $S_W$, we have $G \cdot W^{S_W} = W$. By the claim, 
$\tau(G \cdot X^H) = G \cdot \tau(X^H) = W$.  
Since $G$ acts transitively on the irreducible components of $X$, this implies that $G \cdot X^H$ contains a dense open 
subset $X_0 \subset X$. In other words,
the stabilizer of every point of $X_0$ contains a conjugate of $H$, and part (a) follows.

(b) Clearly $S_X \subset S_W$ and thus $\on{rank}_p(S_X) \leqslant \on{rank}_p(S_W)$. On the other hand, if $S_W$ contains $H = \mu_p^r$ for some $r \geqslant 0$, then
by part (a), $S_X$ also contains a copy of $\mu_p^r$. This proves the opposite inequality, $\on{rank}_p(S_X) \geqslant \on{rank}_p(S_W)$.
\end{proof}

\section{Essential $p$-dimension}
\label{section.ed}

Let $X$ and $Y$ be $G$-varieties. By a correspondence $X \rightsquigarrow Y$ of 
degree $d$ we mean a diagram of rational maps
\[ \xymatrix{ X' \ar@{-->}[d]_{\tiny \text{degree $d$ cover}} \ar@{-->}[dr]^f &   \cr
X & Y \, .} \]
We say that this correspondence is dominant if $f$ is dominant. A rational map 
may be viewed as a correspondence of degree $1$.

The \emph{essential dimension} $\ed(X)$ of a generically free $G$-variety $X$
is the minimal value of $\dim(Y) - \dim(G)$, where the minimum is taken over all generically free
$G$-varieties $Y$ admitting a dominant rational map $X \dasharrow Y$. For a prime integer $p$,
the essential dimension $\ed(X; p)$ of $X$ at $p$ is defined in a similar manner, as $\dim(Y) - \dim(G)$,
where the minimum is taken over all generically free
$G$-varieties $X$ admitting a $G$-equivariant dominant correspondence 
$X \rightsquigarrow Y$ of degree prime to $p$. 

It follows from~\cite[Propositions 2.4 and 3.1]{lotscher2013essential2} that this minimum does not change if
we allow the $G$-action on $Y$ to be $p$-generically free, rather than generically free; we shall not need this 
fact in the sequel. We will, however, need the following lemma.

\begin{lem} \label{lem.projective}
Requiring $Y$ to be projective in the above definitions does not change the values of $\ed(X)$ and $\ed(X; p)$.
That is, for any primitive generically free $G$-variety $X$,

\smallskip
(a) there exists a $G$-equivariant dominant rational map $X \dasharrow Z$, where $Z$ is projective, the $G$-action on $Y$ is generically free,
and $\dim(Y) = \ed(X; G) + \dim(G)$.

\smallskip
(b) There exists a $G$-equivariant dominant correspondence $X \dasharrow Z'$ of degree prime to $p$, 
where $Z'$ is projective, the $G$-action on $Z'$ is generically free,
and $\dim(Z') = \ed(X; p) + \dim(G)$.
\end{lem}

\begin{proof} Let $Y$ be a generically free $G$-variety and $V$ be a generically free linear representation of $G$. It is well known that the $G$-action on $V$ is versal; see, e.g.,~\cite[Proposition 3.10]{merkurjev2013essential}.
Consequently, there exists a $G$-invariant subvariety $Y_1 \subset V$ and a $G$-equivariant dominant rational map $Y \dasharrow Y_1$ so that
the $G$-action on $Y_1$ is generically free. After replacing $Y_1$ by its Zariski closure $Z$ in $\mathbb P(V \oplus k)$, where $G$ acts trivially on $k$, 
we obtain a $G$-equivariant dominant rational map $\alpha \colon Y \dasharrow Z$ such that $Z$ is projective and the $G$-action on $Z$ is generically free.

To prove part (a), choose a dominant $G$-equivariant rational map $f \colon X \dasharrow Y$ such that the $G$-action on $Y$ is generically free and $\dim(Y)$ is the smallest possible, i.e.,
$\dim(Y) = \ed(X)+ \dim(G)$. Now compose $f$ with the map $\alpha \colon Y \dasharrow Z$ constructed above. By the minimality of $\dim(Y)$, we have $\dim(Z) = \dim(Y)$,
and part (a) follows. The proof of part (b) is the same, except that the rational map $f$ is replaced by a correspondence 
of degree prime to $p$.
\end{proof}

%

The essential dimension $\ed(G)$ (respectively the essential dimension at $p$, $\ed(G; p)$) of the
group $G$ is the maximal value of $\ed(X)$ (respectively, of $\ed(X; p)$) taken over all generically free
$G$-varieties $X$.

%
%
%
    
\section{The groups $G_n$}
\label{sect.G_n}

Let $G$ be an algebraic group over $k$ such that the connected component $T = G^0$ is a torus, and the component group
$F = G/T$ is a finite $p$-group, as in~\eqref{e.torus}. 
By~\cite[Lemma 5.3]{lotscher2013essential2}, there exists a finite $p$-subgroup $F' \subset G$ such that $\pi|_{F'}:F'\to F$ is surjective. We will refer 
to $F'$ as a ``quasi-splitting subgroup" for $G$.
We will denote the subgroup generated by $F'$ and $T[n]$ by $G_n$. Here $T[n]$ denotes the $n$-torsion subgroup of $T$,
i.e., the kernel of the homomorphism $\xymatrix{T \ar@{->}[r]^{\times \; \; n} & T}$. Note that our definition of $G_n$ 
depends on the choice of the quasi-splitting subgroup $F'$. We will assume that $F'$ is fixed throughout. We will be particularly interested in the
subgroups
\begin{equation} \label{e.ascending}
G_1 \subset G_p \subset G_{p^2} \subset G_{p^3} \subset \ldots . 
\end{equation}
Informally speaking, we will show that these groups approximate ``$p$-primary behavior" of $G$ in various ways; see~\Cref{lem.finite-subgroup} and~\Cref{prop.finite-subgroup}(b) below.

In the sequel we will denote the center of $G$ by $Z(G)$.

\begin{lem} \label{lem.central}
(a) Let $z \in Z(G)(\overline{k})$ be a central element of $G$ of order $p^n$ for some $n \geqslant 0$.
Then $z \in G_{p^m}$ for $m \gg 0$.

\smallskip
(b) For every $n \geqslant 0$, we have $Z(G)[p^n] = Z(G_{p^r})[p^n]$ as group schemes for all $r \gg 0$. 
\end{lem}

\begin{proof} (a) By the definition of $F'$, there exists and $g \in F'(\cl{k})$ and $t \in T(\cl{k})$ such that $g = z t$.
Since $F'$ is a $p$-group, $g^N = 1$, where $N$ is a sufficiently high power of $p$. Taking $N \geqslant p^n$, we also have $z^N = 1$.
Since $z$ is central, 
$1 = g^N = (zt)^N = z^N t^N = t^N$.
Thus $t \in T[N](\cl{k}) \subset G_{N}(\cl{k})$ and consequently, $z =  g t^{-1}$ is a $\cl{k}$-point of $F' \cdot T[N] = G_N$. 

\smallskip

(b) Let $n\geq 0$ be fixed. Since both $Z(G)[p^n]$ and $G_{p^r}$ are finite $p$-groups, and 
we are assuming that $\on{char}(k) \neq p$, part (a) tells us that there exists $m\geq 0$ such that $Z(G)[p^n] \subset Z(G_{p^r})[p^n]$ as group schemes, for all $r\geq m$. 

Let $r\geq 0$, and let $x\in Z(G_{p^r})[p^n](k_s)$, where $k_s$ is a separable closure of $k$. Let $f_x:T_{k_s}\to T_{k_s}$ be the homomorphism of conjugation by $x$. Passing to character lattices, we obtain a homomorphism $\ang{x}\to \on{GL}_d(\Z)$, where $d=\on{rank}X(T_{k_s})$. By a theorem of Jordan, in $\on{GL}_d(\Z)$ there are at most finitely many finite subgroups, up to conjugacy. In particular, we may find an integer $N\gg 0$ such that the restriction of $\on{GL}_d(\Z)\to \on{GL}_d(\Z/p^N\Z)$ to every finite subgroup is injective. 

Thus, if $r\geq N$, $f_x$ is the identity for every $x\in Z(G_{p^r})[p^n](k_s)$. Since $F'$ is contained in $G_{p^r}$, every $x\in Z(G_{p^r})[p^n](k_s)$ commutes with $F'$. Since $G^0$ and $F'$ generate $G$, we deduce that $x\in Z(G)[p^n](k_s)$. This shows that $Z(G_{p^r})[p^n]\subset Z(G)[p^n]$ for $r\geq N$. We conclude that for $r\geq \max(N,m)$ we have $Z(G_{p^r})[p^n]=Z(G)[p^n]$.
\end{proof} 

\begin{lem} \label{lem.finite-subgroup}
Let $K$ be a $p$-closed field containing $k$. Then every class $\alpha \in H^1(K, G)$ lies in the image
of the map $H^1(K, G_{p^r}) \to H^1(K, G)$ for sufficiently high $r$.
\end{lem}

\begin{proof} Let $\alpha \in H^1(K, G)$.
Consider the following commutative diagram with exact rows
    \begin{equation*}
	\begin{tikzcd}
	1\arrow[r] & T[n] \arrow[r] \arrow[d] & G_n \arrow[r] \arrow[d]  & F \arrow[r] \arrow[d,equal] & 1 \\
	1\arrow[r] & T \arrow[r] & G \arrow[r] & F \arrow[r]&  1 
	\end{tikzcd}
	\end{equation*}
	and the associated diagram in Galois cohomology.
	Let $\cl{\alpha}\in H^1(K,F)$ be the image of $\alpha$ under the natural morphism $H^1(K, G) \to H^1(K, F)$. Since $T$ is abelian,
	the conjugation actions of $G$ on $T$ and of $G_n$ on $T[n]$ descend to $F$. Twisting the bottom sequence by $\cl{\alpha}$, and
	setting $U = \prescript{\cl{\alpha}}{}{T}$, we see that the fiber of $\cl{\alpha}$ equals the image of
	$H^1(K,U)$; see~\cite[Section I.5.5]{serre1997galois}. Similarly twisting the top sequence by $\cl{\alpha}$,
	we see that fiber of $H^1(K,G_n) \to H^1(K,F)$ over $\cl{\alpha}$ 
	equals the image of $H^1(K,U[p^n])$. Hence it suffices to prove the following:
	
	\smallskip
	{\bf Claim:} Let $K$ be a $p$-closed field and $U$ be a torus defined over $K$. Then  
	the natural map $H^1(K,U[p^r]) \to H^1(K, U)$ is surjective for $r$ sufficiently large. 
	
	\smallskip
	To prove the claim, note that since $K$ is $p$-closed, the torus $U$ is split by an extension $L/K$ of degree $n$, where $n$
	is a power of $p$. By a restriction-corestriction argument, it follows that $H^1(K,U)$ is $n$-torsion.
	Now consider the short exact sequence
	\[ \xymatrix{ 1 \ar@{->}[r] & U[n]  \ar@{->}[r] & U  \ar@{->}[r]^{\times \; \; n} & U \ar@{->}[r] &  1. }\]
	The associated exact cohomology sequence
	\[ \xymatrix{ H^1(K, U[n])  \ar@{->}[r] & H^1(K, U) \ar@{->}[r]^{\times \; \; n} & H^1(K, U) }\]
	shows that $H^1(K,U[n])$ surjects onto $H^1(K,U)$.
	This completes the proof of the claim and thus of the Lemma~\ref{lem.finite-subgroup}.
	\end{proof}

\section{The index}
\label{sect.index}

Let $\mu$ be a diagonalizable abelian $p$-group, and
\begin{equation} \label{e.exact-sequence}
\xymatrix{ 1 \ar@{->}[r] & \mu \ar@{->}[r] & G \ar@{->}[r] & \overline{G} \ar@{->}[r] & 1 }
\end{equation}
be a central exact sequence of affine algebraic groups defined over $k$. This sequence 
gives rise to the exact sequence of pointed sets
\[ \xymatrix{  H^1(K, G) \ar@{->}[r] & H^1(K, \overline{G}) \ar@{->}[r]^{\partial_K}        & H^2(K, \mu) } \]
for any field extension $K$ of the base field $k$. Any character
$x \colon \mu \to \bbG_{\on{m}}$, induces a homomorphism $x_* \colon H^2(K, \mu) \to
H^2(K, \bbG_{\on{m}})$.  We define $\on{ind}^{x}(G, \mu)$ as the maximal
index of $x_* \circ \partial_K(E) \in H^2(K, \mu)$, where
the maximum is taken over all field extensions $K/k$ and over all
$E \in H^1(K, \overline{G})$. 
This number is finite for every $x \in X(\mu)$; see~\cite[Theorem 6.1]{merkurjev2013essential}.

\begin{rmk} \label{rem.p-closed} Since $\mu$ is a finite $p$-group, the index of $x_* \circ \partial_K(E)$ does not change
when $K$ is replaced by a finite extension $K'/K$ whose degree is prime to $p$, and $E$ is replaced by its image 
under the natural restriction map $H^1(K, \overline{G}) \to H^1(K', \overline{G})$.
Equivalently, we may replace $K$ by its $p$-closure $K^{(p)}$. In other words, the maximal value
of $x_* \circ \partial_K(E)$ will be attained if we only allow $K$ to range over $p$-closed fields 
extensions of $k$. 
\end{rmk}

Set $\on{ind}(G, \mu) : = \min \; \sum_{i= 1}^r \, \on{ind}^{x_i}(G, \mu) $, where the minimum is taken over all   
generating sets $x_1, \dots, x_r$ of the group $X(\mu)$ of characters of $\mu$.

Now suppose $G^0 = T$ is a torus, and $G/G^0 = F$ is a $p$-group, as in ~\eqref{e.torus}. In this case there is a particularly convenient choice of
$\mu \subset G$. Following~\cite[Section 4]{lotscher2013essential2} we will denote this central subgroup of $G$ by $C(G)$.
If $k$ is algebraically closed, $C(G)$ is simply the $p$-torsion subgroup of the center of $G$, $C(G) = Z(G)[p]$. If $k$ is only assumed to be $p$-closed, then
we set $\mu = \on{Split}_k(Z(G)[p])$ to be the largest $k$-split subgroup of $Z(G)[p]$ in the sense of~\cite[Section 2]{lotscher2013essential}.

\begin{prop} \label{prop.finite-subgroup} Let $G$ be as in~\eqref{e.torus}. Denote by $\eta(G)$ the smallest dimension of a $p$-faithful $G$-representation. 
Then:

\smallskip
(a) $\on{ind}(G, C(G)) = \eta(G)$.

\smallskip
(b) If $r$ is sufficiently large, then 
$\eta(G) = \eta(G_{p^r}) = \ed(G_{p^r}) = \ed(G_{p^r}; p)$.
\end{prop}

\begin{proof} (a) Let $\on{Rep}^x(G)$ be the set of irreducible $G$-representations $\nu \colon G \to \GL(V)$ such that $\nu(z) = x(z) \on{Id}_V$ for every $z \in \mu(\cl{k})$.
By the Index Formula~\cite[Theorem 6.1]{merkurjev2013essential}, 
$\on{ind}^x(G) = \on{gcd} \, \dim(\nu)$, where $\nu$ ranges over $\on{Rep}^x(G)$, and $\on{gcd}$ stands for the greatest common divisor.
By \cite[Proposition 4.2]{lotscher2013essential2}, 
$\dim(\nu)$ is a power of $p$ for every irreducible representation $\nu$ of $G$ defined over $k$. Thus one can replace $\on{gcd} \, \dim(\nu)$ 
by $\min \, \dim(\nu)$
in the Index Formula. Decomposing an arbitrary representation of $G$ as a direct sum of irreducible subrepresentations, we see that
$\on{ind}(G, C(G))$ = minimal dimension of a $k$-representation $\nu \colon G \to \GL(V)$ such that the restriction 
$\nu_{| C(G)} \colon C(G) \to \GL(V)$ is faithful. 
Finally, by \cite[Proposition~4.3]{lotscher2013essential2}, $\nu_{| C(G)}$ is faithful if and only if $\nu$ is $p$-faithful.

\smallskip
(b) Since $G_{p^r}$ is a (not necessarily constant) finite $p$-group and $k$ is $p$-closed, the identities  
$\eta(G_{p^r}) = \ed(G_{p^r}) = \ed(G_{p^r}; p)$ follow from \cite[Theorem 7.1]{lotscher2013essential}. It thus remains to show that
\begin{equation} \label{e.eta}
\text{$\eta(G) = \eta(G_{p^r})$ for $r \gg 0$.}
\end{equation}
By \Cref{lem.central}(b), $Z(G)[p] = Z(G_{p^r})[p]$ and thus $C(G) = C(G_{p^r})$ for $r \gg 0$.
In view of part (a), \eqref{e.eta} is thus equivalent to 
\begin{equation} \label{e.index}
\text{$\on{ind}(G, C(G)) = \on{ind}(G_{p^r}, C(G))$ 
for $r \gg 0$.}
\end{equation}
Let $h$ be the natural projection $G \to \overline{G} = G/C(G)$. Note that the group $\overline{G}$ is of the same type as $G$. That is,
the connected component $\overline{G}^0$ is the torus $\overline{T} := h(T)$, and since
the homomorphism $F = G/T \to \overline{G}/\overline{T}$ is surjective, $\overline{F} : = \overline{G}/\overline{G}^0$ is a $p$-group.
Moreover, if $F'$ is a quasi-splitting subgroup for $G$ (as defined at the beginning of \Cref{sect.G_n}), then $\overline{F}' :=h(F')$ 
is a quasi-splitting subgroup for $\overline{G}$.
We will use this subgroup to define the finite subgroups $\overline{G}_n$ of $\overline{G}$ for every integer $n$ in the same way as we defined $G_n$:
\[ \text{$\overline{G}_n$ is the subgroup of $\overline{G}$ generated by $\overline{F}'$ and torsion subgroup $\overline{T}[n]$.} \]
Now observe that since $C(G)$ is $p$-torsion in $G$, $h(T[n]) \subset \overline{T}[n] \subset h(T[pn])$ and thus 
\begin{equation} \label{e.inclusion}
h(G_n) \subset \overline{G}_n \subset h(G_{pn}). 
\end{equation}
for every $n$. We now proceed with the proof of~\eqref{e.index}. Consider the diagram of natural maps
\[  \xymatrix{  1 \ar@{->}[r] & C(G) \ar@{->}[r] \ar@{=}[d]  & G \ar@{->}[r] & \overline{G} \ar@{->}[r] & 1 \\  
1 \ar@{->}[r] & C(G) \ar@{->}[r]  & G_{p^r} \ar@{->}[r] \ar@{^{(}->}[u]^i & h({G}_{p^r}) \ar@{->}[r] \ar@{^{(}->}[u]^{\overline{i}} & 1 ,}
\]
and the induced diagram in Galois cohomology
\[  \xymatrix{  H^1(K, G) \ar@{->}[r] & H^1(K, \overline{G}) \ar@{->}[r]^{\partial_K}        & H^2(K, C(G)) \ar@{=}[d]  \\ 
H^1(K, G_{p^r}) \ar@{->}[u]^{i_*} \ar@{->}[r] & H^1(K, h(G_{p^r})) \ar@{->}[u]^{\overline{i}_*} \ar@{->}[r]^{\overline{\partial}_K} & H^2(K, C(G)).} \]
In view of Remark~\ref{rem.p-closed}, for the purpose of computing $\on{ind}(G, C(G))$ and $\on{ind}(G_{p^r}, C(G))$, we may assume that
$K$ is a $p$-closed field. We claim that for $r \gg 0$, the vertical map $\overline{i}_* \colon H^1(K, h(G_{p^r})) \to H^1(K, \overline{G})$ 
is surjective for every $p$-closed field $K/k$. If we can prove this claim, then for $r \gg 0$, the image of $\overline{\partial}_K$ 
in $H^2(K, C(G))$ is the same as the image of $\partial_K$. Thus $\on{ind}^x(G)$ and $\on{ind}^x(G_{p^r})$ are the same for every $x \in X(C(G))$,
and~\eqref{e.index} will follow.

To prove the claim, note that by~\eqref{e.inclusion}, $\overline{G}_{p^r} \subset h(G_{p^{r+1}})$. Consider the composition
\[  \xymatrix{  H^1(K, \overline{G}_{p^{r-1}}) \ar@{->}[r] & H^1(K, h(G_{p^r})) \ar@{->}[r]^{\quad \overline{i}_*} & H^1(K, \overline{G}) .} \]
By~\Cref{lem.finite-subgroup}, the map $H^1(K, \overline{G}_{p^{r-1}}) \to H^1(K, \overline{G})$ is surjective for $r \gg 0$.
Hence, so is $\overline{i}_*$. This completes the proof of the claim and thus of~\eqref{e.index} and of~\Cref{prop.finite-subgroup}.
\end{proof}

\section{A resolution theorem for rational maps}
\label{sect.resolution}

The following lemma is a minor variant of~\cite[Lemma 2.1]{brv-mixed}. For the sake of completeness, we supply a self-contained proof. 

\begin{lemma}\label{trdeg} 
Let $K\c L$ be a field extension and $v:L^{\times}\to \Z$ be a discrete valuation. Assume that $v|_{K^{\times}}$ is non-trivial and denote the residue fields of $v$ and $v|_{K^{\times}}$ 
by $L_v$ and $K_v$, respectively. Then $\trdeg_K L \geqslant \trdeg_{K_v}L_v$.
\end{lemma}

\begin{proof}
    Let $\cl{x}_1,\dots,\cl{x}_m \in L_v$. For every $i$, let $x_i$ be a preimage of $\cl{x}_i$ in the valuation ring $\mc{O}_L$. 
    It suffices to show that $\cl{x}_1,\dots,\cl{x}_m$ are algebraically independent over $K_v$, then $x_1,\dots,x_m$ are algebraically independent over $K$. 
    To prove this, we argue by contradiction. Suppose there exists a non-zero polynomial
    $f\in K[t_1,\dots,t_m]$ such that $f(x_1,\dots,x_m)=0$. Multiplying $f$ by a suitable power of a uniformizing parameter for $v|_{K^{\times}}$, we may assume that $f\in \mc{O}_K[x_1,\dots,x_m]$ and that at least one coefficient of $f$ has valuation equal to $0$. 
    Reducing modulo the maximal ideal 
    of the valuation ring $\mc{O}_K$, we see that $x_1,\dots,x_m$ are algebraically dependent over $K_v$, a contradiction.
\end{proof}
    
Recall that if $X_1$ is regular in codimension $1$ (e.g. $X_1$ is normal) and $X_2$ is complete, any rational map $f:X_1\dashrightarrow X_2$ is regular in codimension $1$. It follows that if $D\c X_1$ is a prime divisor of $X_1$, the closure of the image $\cl{f(D)}\c X_2$ is well-defined. 
    
\begin{thm}\label{divisorial} Let $G$ be a linear algebraic group over $k$,
and $f \colon X \dasharrow Y$ be a dominant rational map of $G$-varieties.
Assume that $Y$ is complete, $D \subset X$ is a prime divisor, 
and $\cl{f(D)}\neq Y$. Then there exist a commutative diagram of $G$-equivariant dominant rational maps 
\[  \xymatrix{  & Y' \ar@{->}[d]^{\pi} \\  
X \ar@{-->}[r]_f  \ar@{-->}[ur]^{f'}   & Y ,}
\]
and a divisor $E \subset Y'$ such that $Y'$ is normal and complete, $\pi \colon Y' \to Y$ is a birational morphism, and $\cl{f'(D)}=E$.
\end{thm}

\begin{proof}
    Let $v: k(X)^{\times}\to \Z$ be the valuation given by the order of vanishing or pole along $D$. Define $C:=\cl{f(D)}$ and let $w \colon k(Y)^{\times} \stackrel{f^*}{\longrightarrow} k(X)^{\times} \stackrel{v}{\longrightarrow} \Z$. Let $\phi\in k(Y)^{\times}$ be such that $f$ is regular in an open neighbourhood $U$ of the generic point of $C$, and such that $\phi|_{U\cap C}=0$. It follows that $\phi\circ f$ is zero on $D$, hence $w(f) = v(\phi\circ f)>0$. This shows that $w$ is non-zero, and so $w$ is a discrete valuation on $k(Y)$. 
    
    Since $D$ maps dominantly onto $C$, we have an inclusion of local rings $f^*:\mc{O}_{Y,C}\hookrightarrow \mc{O}_{X,D}$. It follows that if $\phi\in \mc{O}_{Y,C}$, then $w(\phi)=v(\phi\circ f)\geq 0$, i.e. $\mc{O}_{Y,C}$ is contained in the valuation ring of $w$. In other words, $C$ is the center of $w$.
    
    Denote by $k(Y)_w$ the residue field of $w$. By \Cref{trdeg} we have 
    \[\trdeg_k k(X)-\trdeg_k k(Y)\geq \trdeg_k k(D)-\trdeg_k k(Y)_w.\]
    Since $\trdeg_k k(D)=\trdeg_kk(X)-1$, we obtain that
    $\trdeg_k k(Y)_w \geq \trdeg_k k(Y)-1$. By the Zariski-Abhyankar inequality \cite[VI, \S 10.3, Cor 1]{bourbaki1989commutative} we have $\trdeg_k k(Y)_w \leqslant \trdeg_k k(Y)-1$, hence \[\trdeg_k k(Y)_w=\trdeg_k k(Y)-1.\]
    By \cite[Theorem 5.2]{artin1986neron}, there exists a sequence of proper birational morphisms \[Y' = Y_n\to Y_{n-1}\to\dots\to Y_1\to Y_0 = Y \] such that $Y_{i+1}\to Y_i$ is a blow-up at the center of $w$ on $Y_i$, and such that the center $E'$ of $w$ on $Y'$ is a prime divisor and $Y'$ is normal at the generic point of $E'$. Since $C$ is $G$-inavriant,
    by the universal property of the blow-up, the $G$-action on $Y$ lifts to every $Y_i$, and the maps $Y_{i+1}\to Y_i$ are $G$-equivariant. 
    
    We let $\pi: Y'\to Y$ be the composition of the maps $Y_{i+1}\to Y_i$, and $f': X \to Y'$ be the composition of $f$ with the birational inverse of $\pi$. By construction, $f'$ is $G$-equivariant. It suffices to show that $\cl{f'(D)}=E$. Since the center of $w$ is the divisor $E\c Y'$, the valuation $w$ is given by the order of vanishing or pole along $E$. If we identify $k(Y)$ with $k(Y')$ via $\pi$, we also have $w=(f')^*v$. It follows that for every $\phi \in k(Y')^{\times}$, $\phi$ is regular and vanishes at the generic point of $E$ if and only if $w(\phi)>0$ if and only if $v(\phi\circ f')=0$ if and only if $\phi$ vanishes at the generic point of $f'(D)$. We conclude that $\cl{f'(D)}=E$, as desired. Finally, after replacing $Y'$ by its normalization,
    $(Y')^n$ and $E'$ by its preimage in $(Y')^n$, we may assume that $Y'$ is normal everywhere (and not just at a generic point of $E'$). The $G$-action naturally 
    lifts to $(Y')^n$.
\end{proof}

\section{Proof of Theorem~\ref{thm.main}}
\label{sect.proof}
Let $G$ be an algebraic group as in (\ref{e.torus}). Let $V$ be a $p$-faithful representation of $G$ of minimal  
dimension $\eta(G)$. By~Lemma~\ref{lem.sgp-existence} there exists a stabilizer in general position $S_V$ for the $G_{\cl{k}}$-action on $V_{\cl{k}}$. Since $V(k)$ is dense in $V$, we may assume without loss of generality that $S_V$ is the stabilizer of a $k$-point of $V$. In particular, we may assume that $S_V$ is a closed subgroup of $G$ defined over $k$.
Since $T$ acts $p$-faithfully on $V$, we have $S_V \cap T = \set{1}$. 

\begin{reduction} \label{red8.1}
To prove \Cref{thm.main}, it suffices to construct a $G$-representation $V'$ such that $\on{dim}(V')=\on{rank}_p(S)$, $W:=V\oplus V'$ is $p$-generically free, and
\begin{equation} \label{e.suffices}
\ed(W; p) = \dim(W) - \dim(G). 
\end{equation}
\end{reduction}

Here when we write $\ed(W; p)$, we are viewing $W$ as a generically free $G/\on{Ker}(\phi)$-variety, were $\phi \colon G \to W$ denotes the representation of $G$ on $W$. The kernel, $\on{Ker}(\phi)$, of this representation is a finite normal subgroup of $G$ of
order prime to $p$. 

\begin{proof} Suppose we manage to construct $V'$ so that~\eqref{e.suffices} holds. Then
\[ \ed(W; p) \stackrel{(i)}{=} \ed(G/\on{Ker}(\phi) \, ; \, p) \stackrel{(ii)}{=} \ed(G; p) \stackrel{(iii)}{\leqslant} \rho(G) - \dim(G) \stackrel{(iv)}{\leqslant} \dim(W) - \dim(G),\]
where 

\smallskip
(i) follows from the fact that $W$ is a versal $G/\on{Ker}(\phi)$-variety; see, e.g.,~\cite[Propositions 3.10 and 3.11]{merkurjev2013essential},

\smallskip
(ii) by \cite[Proposition 2.4]{lotscher2013essential2},

\smallskip
(iii) is the right hand side of~\eqref{e.lmmr}, and

\smallskip
(iv) is immediate from the definition of $\rho(G)$.

\smallskip
\noindent
If we know that~\eqref{e.suffices} holds, then the inequalities (iii) and (iv) are, in fact, equalities. Equality in (iii) yields Theorem~\ref{thm.main}(a). On the other hand, since \[\on{dim}(W)=\on{dim}(V)+\on{dim}(V')=\eta(G)+\on{rank}_p(S),\] equality in (iv) tells us that $\eta(G)+\on{rank}_p(S) = \rho(G)$, thus proving Theorem~\ref{thm.main}(b).
\end{proof}

To construct $W$, we begin with a $p$-faithful linear representation $\nu \colon G \to \on{GL}(V)$ of minimal 
possible dimension $d= \eta(G)$. The kernel of $\nu$ is a finite group of order prime to $p$; it is contained in 
the maximal torus $T$ of $G$. From now on we will replace $G$ by $\overline{G} = G/\on{Ker}(\nu)$. All other $G$-actions 
we will construct (including the linear $G$-action on $W$) will factor through $\overline{G}$. In the end we will show that
$\ed(W; p) = \ed(\overline{G}; p)$; once again, this is enough because $\ed(G; p) = \eta(G) = \eta(\overline{G}) = \ed(\overline{G}; p)$ 
by~\cite[Proposition 2.4]{lotscher2013essential2}. In other words, from now on we may (and will) assume that the $G$-action on $V$ is faithful.

Recall that $S_V$ denotes the stabilizer in general position for the $G$-action on $V$, and that we have chosen $S_V$ (which is a priori a closed subgroup of $S_{\overline{k}}$ defined up to conjugacy), so that it is defined over $k$. Since $T$ is a torus, and $T$ acts faithfully on $V$, this action is automatically generically free.  That is $S_V \cap T = 1$ or equivalently, the natural projection $\pi|_{S_V}:S_V\to F$ is injective. In particular, $\pi(S_V)$ is diagonalizable. 
By our assumption $F$ is isomorphic to $\mu_{p^{i_1}} \times \dots \times \mu_{p^{i_R}}$ 
for some integers $R \geqslant 0$ and $i_1, \ldots, i_R \geqslant 1$. 
Moreover, this isomorphism can be chosen so that
\[ \pi(S_V) = \mu_{p^{j_1}} \times \dots \times \mu_{p^{j_r}} \]
for some $0 \leqslant r \leqslant R$ and some integers $1 \leqslant j_t \leqslant i_t$, for every $t = 1, \dots, r$.  
Let $\chi_t$ be the composition of $\pi \colon G \to F$ with the projection map $F \to \mu_{p^{i_t}}$ to the $t$-th component
and $V_t$ be a $1$-dimensional vector space on which $G$ acts by $\chi_t$. 
Set $W_d = V$ and $W_{d + t} = V \oplus V_1 \oplus \dots \oplus V_t$ for $m = 1, \dots, r$.
A stabilizer in  general position for the $G$-action on $W_{d + m}$ is clearly 
\[
S_{W_{d + m}} = S_V \cap \on{Ker}(\chi_1) \cap \dots \cap \on{Ker}(\chi_m) \]
and thus
\begin{equation} \label{e.p-rank-S_V}
 S_{W_{d + m}} \simeq \pi(S_{W_{d + m}}) = \{ 1 \} \times \dots \times \{ 1 \} \times\mu_{p^{j_{m +1}}} \times \dots \times
\mu_{p^{j_{d + r}}} \end{equation}
for any $0 \leqslant m \leqslant r$. 
In particular, $S_{W_{d + r}} = \{ 1 \}$, in other words,
the $G$-action on $W_{d + r}$ is generically free.
We now set 
\[ W = W_{d + r} = V \oplus V_1 \oplus \dots \oplus V_r. \]
Having defined $W$, we now proceed with the proof of~\eqref{e.suffices}.
In view of \Cref{lem.projective}(b) it suffices to establish the following.

\begin{prop} \label{prop.main}
Let $W$ be as above. Consider a dominant $G$-equivariant correspondence
\[ \xymatrix{ X \ar@{-->}[d]_{\tau} \ar@{-->}[dr]^{f} &   \cr W & Y, } \]
of degree prime to $p$, where $Y$ is a $p$-generically free projective $G$-variety.
Then $\dim(Y) = \dim(W) = d + r$.
\end{prop}

We now proceed with the proof of the proposition. By~\Cref{lem.covers} (with $Z = W_{ d + r - 1}$) 
there exists a commutative diagram of $G$-equivariant maps
\[ \xymatrix{   D_{d + r - 1} \ar@{->}[dd]_{\tau_{d + r - 1}} \ar@{^{(}->}[r] & X_{d + r} \ar@{-->}[d]^{\alpha_{d + r}} \cr 
                                                             & X \ar@{-->}[d]^{\tau}  \cr
                                           W_{d + r - 1} \ar@{^{(}->}[r] & W } \]
such that $X_{d + r}$ is normal, $\alpha_{d + r}$ is a birational isomorphism, 
$D_{d + r - 1}$ is an irreducible divisor in $X_{d + r}$, and $\tau_{d + r - 1}$ is a cover of $W_{d + r - 1}$ of degree prime to $p$.
Let $S_{D_{d + r - 1}} \subset G$ be a stabilizer in  general position for
the $G$-action on $D_{d + r - 1}$; it exists by~Lemma~\ref{lem.sgp-existence}. In view of \eqref{e.p-rank-S_V}, 
\Cref{lem.sgp-in-covers} tells us that
\begin{equation} \label{e.p-rank-S_X}
\text{$\on{rank}_p(S_{D_{d + r - 1}}) = 1$.} 
\end{equation}
On the other hand, by our assumption the $G$-action on $Y$ is $p$-generically free. Thus
the restriction of $f$ (viewed as a dominant rational map $X_{d + r} \dasharrow Y$) to $D_{d + r - 1}$ cannot be dominant, and~\Cref{divisorial} applies: there exists a commutative diagram 
\[  \xymatrix{ X_{d + r} \ar@{-->}[d]_{\alpha_{d + r}} \ar@{-->}[r]^{\quad f_{d + r} \quad} & Y_{d+r} \ar@{->}[d]^{\sigma_{d+r}} \\  
                                          X  \ar@{-->}[r]^f   & Y ,}
\]
of dominant $G$-equivariant rational maps, where $\sigma_{d+r}$ is a birational morphism,
$Y_{d+r}$ is normal and complete, and $f_{d + r}$ restricts to a dominant $G$-equivariant rational map
$D_{d + r-1} \dasharrow E_{d + r-1}$ for some $G$-invariant irreducible divisor $E_{d + r-1}$ of $Y_{d+r}$. 
We will denote this dominant rational map by $f_{d + r -1} \colon D_{d + r - 1} \dasharrow E_{d + r -1}$.
We now iterate this construction with $f_{d + r}$ replaced by $f_{d + r - 1}$.

By~\Cref{lem.covers} there exists a commutative diagram of $G$-equivariant maps
\[ \xymatrix{   D_{d + r - 2} \ar@{->}[dd]_{\tau_{d + r - 2}} \ar@{^{(}->}[r] & X_{d + r-1} \ar@{-->}[d]^{\alpha_{d + r-1}} \cr 
                                                             & D_{d + r-1} \ar@{-->}[d]^{\tau_{d + r -1}}  \cr
                                           W_{d + r - 2} \ar@{^{(}->}[r] & W_{d + r - 1} } \]
such that $X_{d + r - 1}$ is normal, $\alpha_{d + r-1}$ is a birational isomorphism, 
$D_{d + r - 2}$ is an irreducible divisor in $X_{d + r-1}$, and $\tau_{d + r - 2}$ is a cover of $W_{d + r - 2}$ of degree prime to $p$.

Denote a stabilizer in general position for the $G$-action on $E_{d + r -1}$ by
$S_{E_{d + r -1}}$. Recall that the $G$-action on $Y$ (and thus $Y_{d + r}$) is $p$-generically free.
Since $E_{d + r - 1}$ is a $G$-invariant hypersurface in $Y_{d + r}$,~\Cref{lem.p-rank}(a) tells us that
$\on{rank}_p(S_{E_{d + r-1}}) \leqslant 1$. On the other hand, 
since $X_{d + r -1}$ maps dominantly to $E_{d + r -1}$, $S_{E_{d + r-1}}$ contains 
(a conjugate of) $S_{X_{d + r -1}}$ and
thus $\on{rank}_p(S_{E_{d + r -1}}) \geqslant \on{rank}_p(S_{X_{d + r -1}})$, where
$\on{rank}_p(S_{X_{d + r -1}}) = 1$ by~\eqref{e.p-rank-S_X}. We conclude that $\on{rank}_p(S_{E_{d + r-1}}) = 1$.
Now observe that since $\on{rank}_p(S_{E_{d+r-1}}) = 1$ and $\on{rank}_p(S_{X_{d + r-2}}) = 2$ (see~\eqref{e.p-rank-S_V}),
$f_{d+r-1}(X_{d +r -2})$ cannot be dense in $E_{d + r-1}$. Consequently,~\Cref{divisorial} can be applied to 
$f_{d+r-1} \colon X_{d+r-1} \dasharrow E_{d+r-1}$. It yields 
a birational morphism $\sigma_{d+r-1} \colon Y_{d+r-1} \to E_{d+r-1}$ such that $Y_{d+r-1}$ is normal and complete, and the composition $\sigma_{d+r-1}^{-1} \circ f_{d+r-1}$ restricts to a dominant $G$-equivariant rational map $f_{d+r-2} \colon D_{d+r-2} 
\dasharrow E_{d+r-2}$ for some $G$-invariant prime divisor $E_{d+r-2}$ of $Y_{d+r -1}$.
Proceeding recursively, we obtain a commutative diagram of $G$-equivariant maps
\[  \xymatrix@R=.8cm@C=.4cm{ 
X_d \ar@{-->}[d]_{\alpha_d}  \ar@{-->}[rrrrrrrrr]^{f_d} & & & &  & &  &  &  &  Y_d \ar@{->}[d]^{\sigma_d}
\cr
D_{d} \ar@{^{(}->}[r] \ar@{->}[dddddd]_{\tau_d} & X_{d+1} \ar@{-->}[rrrrrrr]^{f_{d+1}} \ar@{-->}[d]_{\alpha_{d+1}} & & & & & & & Y_{d+1} \ar@{->}[d]^{\sigma_{d+1}} &  \ar@{_{(}->}[l] E_{d} &  
\cr
 & \dots & \dots & \dots & \dots   &  \dots     &    \dots       &    \dots     &  \dots  & 
\cr
& &  \dots  \ar@{->}[d] &  \dots   &  \dots &  \dots     &     \dots       &    \dots \ar@{->}[d]^{\sigma_{d+r-2}}    &  
  &     
\cr
& & D_{d + r - 2} \ar@{^{(}->}[r] & X_{d + r -1} \ar@{-->}[rrr]^{f_{d+r-1}} \ar@{-->}[d]_{\alpha_{d + r -1}} & &      & Y_{d+r-1} \ar@{->}[d]^{\sigma_{d+r-1}} &  \ar@{_{(}->}[l] E_{d+r-2} &      &
\cr
& &  & D_{d + r - 1} \ar@{->}[dd]_{\tau_{d+r - 1}} \ar@{^{(}->}[r] & X_{d+r} \ar@{-->}[d]_{\alpha_{d+r}}
\ar@{-->}[r]^{{\small f_{d+r}}} & Y_{d+r} \ar@{->}[d]^{\sigma_{d+r}}  & \ar@{_{(}->}[l] E_{d+r-1}        &     &  &  
\cr
 &  &  & &    X \ar@{-->}[r]^{f}   \ar@{-->}[d] &  Y                &                               &     & &  
 \cr
 W_d  \ar@{^{(}->}[r] & W_{d+1}  \ar@{^{(}->}[r] & \ldots \ar@{^{(}->}[r] & W_{d + r-1} \ar@{^{(}->}[r] & W_{d + r} & & & & & &  } \]
such that for every $m$, we have

\smallskip
(i) $D_{d + m}$ is an irreducible divisor in $D_{d + m}$ and  $E_{d + m - 1}$ is an irreducible divisor in $Y_{d + m}$,

\smallskip
(ii) the vertical maps $\alpha_{d + m}$ and $\sigma_{d+m}$ are birational isomorphisms,
 
\smallskip
(iii) $X_{d + m}$ and $Y_{d + m}$ are normal and $Y_{d + m}$ is complete,

\smallskip
(iv) $\on{rank}_p(S_{X_{d + m}}) = \on{rank}_p(S_{Y_{d + m}}) = r - m$,

\smallskip
(v) $\tau_{d + m}$ is a cover of degree prime to $p$.

\smallskip
\noindent
Note that the subscripts are chosen so that $\dim(X_{d + m}) = \dim(W_{d + m}) = d + m$, for each $m = 0, \ldots, r$. 
We will eventually show that $\dim(Y_{d + m}) = d+m$ for each $m$ as well, but we do not know it at this point. 

\begin{lem} \label{lem.p-faithful} The $G$-action on $Y_{d + m}$ (or equivalently, on $E_{d + m}$) 
is $p$-faithful for every $m = 0, \ldots, r$.
\end{lem}

Assume, for a moment, that this lemma is established. By our construction $f_d$ may be viewed as a dominant $G$-equivariant
correspondence $W_d \rightsquigarrow Y_d$ of degree prime to $p$. Now recall that $W_d = V$ is a $p$-faithful representation of
$G$ of minimal possible dimension $\eta(G)$. By Lemma~\ref{lem.p-faithful}, the $G$-action of $Y_d$ 
is $p$-faithful. Restricting to the $p$-subgroup $G_{n} \subset G$, where $n$ is a power of $p$,
we obtain a dominant $G_n$-equivariant correspondence $f_d \colon V \rightsquigarrow Y_d$ of degree prime to $p$,
where the $G_n$-action on $Y$ is faithful. Thus $\dim(Y_d) \geqslant \ed(G_n; p)$.
When $n$ is a sufficiently high power of $p$, Proposition~\ref{prop.finite-subgroup} tells us that 
\[ \ed(G_n; p) = \eta(G_n) = \eta(G) = \dim(V) = d. \] 
By conditions (i) and (ii) above, $\dim(Y_{d + m +1}) = \dim(E_{d + m}) + 1 = \dim(Y_{d + m}) + 1$ for each $m = 0, 1, \dots, r$. 
Thus
$\dim(Y) = \dim(Y_{d + r}) = \dim(Y_d) + r = \dim(V) + r = d + r = \dim(W)$, as desired. 
This will complete the proof of Proposition~\ref{prop.main} and thus of Theorem~\ref{thm.main}.

\begin{proof}[Proof of Lemma~\ref{lem.p-faithful}] For the purpose of this proof, we may replace $k$ by its algebraic closure $\overline{k}$
and thus assume that $k$ is algebraically closed. We argue by reverse induction on $m$. 
For the base case, where $m = r$, note that by our assumption the $G$-action on $Y$ is $p$-faithful. 
Since $Y_{d + r}$ is birationally isomorphic to $Y$, the same is true of the $G$-action on $Y_{d + r}$.

For the induction step, assume that the $G$-action on $Y_{d + m + 1}$ is $p$-faithful for some $0 \leqslant m \leqslant r - 1$. 
Our goal is to show that the $G$-action on $Y_{d + m}$ is also $p$-faithful. Let $N$ be the kernel of the $G$-action on $Y_{d+m}$.
Recall that by Lemma~\ref{lem.p-rank}(b), there is a homomorphism 
\begin{equation} \label{e.N} \alpha \colon N \to \bbG_{\on{m}} 
\end{equation}
where $\on{Ker}(\alpha)$ has no elements of order $p$.
Since $\on{Ker}(\alpha)$ is a subgroup of $G$, and we are assuming that $G^0 = T$ 
is a torus and $G/G^0 = F$ is a finite $p$-group, we conclude that
\begin{equation} \label{e.K}
\text{$\on{Ker}(\alpha)$ is a finite subgroup of $T$ of order prime to $p$.}
\end{equation}
It remains to show $\alpha(N)$ is a finite group of order prime to $p$. Assume the contrary: 
$\alpha(N)$ contains $\mu_p \subset \bbG_{\on{m}}$.

\smallskip
{\bf Claim:} 
There exists a subgroup $\mu_p \simeq N_0 \subset N$ such that $N_0$ is central in $G$.

\smallskip
Since $G^0 = T$ is a torus and $G/G^0 = F$ is a $p$-group, if $N_0 \simeq \mu_p$ is normal in $G$, then
the conjugation map $G \to \on{Aut}(\mu_p) \simeq \mathbb Z/ (p-1) \bbZ$ is trivial, so $N_0$ is automatically central.
Thus in order to prove the claim, it suffices to show that
there exists a subgroup $\mu_p \simeq N_0 \subset N$ such that $N_0$ is normal in $G$.
Now consider two cases. 

Case 1: $G^0 = T$ does not act $p$-faithfully on $Y_{d + m}$. 
Then $\mu_p \subset N \cap T \triangleleft G$. In view of~\eqref{e.N} and \eqref{e.K}, $N \cap T$ contains exactly one copy of $\mu_p$.
This implies that $\mu_p$ is characteristic in $N \cap T$ and hence, normal in $G$, as desired.

Case 2: $N \cap T$ does not contain $\mu_p$, i.e., $N \cap T$ is a finite group of order prime to $p$.
Examining the exact sequence
\[ 1 \to N \cap T \to N \to F = G/T \]
we see that $N$ is a finite group of order $pm$, where $m$ is prime to $p$. Let $\on{Syl}_p(N)$ be the set of Sylow $p$-subgroups of $N$. 
By Sylow's theorem $|\on{Syl}_p(N)| \equiv 1 \pmod{p}$.~\footnote{Recall that we are assuming that $k$ is an algebraically closed field of characteristic $\neq p$. If $\on{char}(k)$
does not divide $|N|$, $\on{Syl}_p(N)$ is the set of Sylow subgroups of the finite group $N(k)$. If $\on{char}(k)$ divides $|N|$, then
elements of $\on{Syl}_p(N)$ can be identified with Sylow $p$-subgroups of the finite group $N_{red}(k)$.} 
The group $G$ acts on $\on{Syl}_p(N)$ by conjugation. Clearly $T$ acts trivially, and the $p$-group
$F = G/T$ fixes a subgroup $N_0 \in \on{Syl}_p$. In other words, 
$N_0 \simeq \mu_p$ is normal in $G$. This proves the claim.

\smallskip
We are now ready to finish the proof of Lemma~\ref{lem.p-faithful}.
Let $S_{Y_{d + m}} \subset G$ be a stabilizer in  general position for the $G$-action on $Y_{d + m}$. 
Clearly $N_0 \subset N \subset S_{Y_{d + m}}$. Since $f_{d + m} \colon X_{d + m} \dasharrow Y_{d + m}$ is a dominant
$G$-equivariant rational map, $S_{Y_{d + m}}$ contains (a conjugate of) $S_{X_{d + m}}$. By (iv)
\begin{equation} \label{e.p-rank-S_Y}
\on{rank}_p (S_{Y_{r + m}}) = r- m  = \on{rank}_p(S_{X_{d + m}}). \end{equation}
In particular, $S_{X_{d + m}}$ contains a subgroup $A$ isomorphic to $\mu_p^{r -m}$.
Since $N_0 \simeq \mu_p$ is central in $G$, it has to be contained in $A$; otherwise, $S_{Y_{n-m}}$ would contain
a subgroup isomorphic to $A \times \mu_p = (\mu_p)^{r - m + 1}$, contradicting~\eqref{e.p-rank-S_Y}. Thus
$\mu_p \simeq N_0 \subset {S}_{X_{d + m}}$. Moreover, since $N_0$ is normal in G, it is contained in every 
conjugate of $S_{X_{d + m}}$. 
This implies that $N_0$ stabilizes every point of $X_{d + m}$. We conclude that $N_0$ acts
trivially on $X_{d + m}$ and hence on $X_d \subset X_{d + m}$ and on
$\tau_d(X_d) = W_d = V$. This contradicts our assumption that $G$ acts $p$-faithfully on $W_d = V$.

This contradiction shows that our assumption that $\alpha(N)$ contains $\mu_p$ was false. Returning to~\eqref{e.N} and~\eqref{e.K}, we conclude that the kernel $N$ of the $G$-action on $Y_{d + m}$ is 
a finite group of order prime to $p$. In other words, the $G$-action on $Y_{d + m}$ is $p$-faithful. This
completes the proof of Lemma~\ref{lem.p-faithful} and thus of Proposition~\ref{prop.main} and Theorem~\ref{thm.main}.
\end{proof}


\begin{rmk} \label{rem1} Our proof of Theorem~\ref{thm.main} goes through even if $F$ is not abelian, provided that
the stabilizer in general position $S_V$ projects isomorphically to $F/[F, F]$. (If $F$ is abelian,
this is always the case.)
\end{rmk}

\section{Normalizers of maximal tori in split simple groups}
\label{sect.normalizers}

In this section $\Gamma$ will denote a split simple algebraic group over $k$, $T$ will denote a $k$-split maximal torus of $\Gamma$, 
$N$ will denote the normalizer of $T$ in $\Gamma$, and $W = N/T$ will denote the Weyl group. These groups fit into an exact sequence
\begin{equation} \label{e.weyl}
\xymatrix{ 1 \ar@{->}[r] &  T \ar@{->}[r] & N \ar@{->}[r]^{\pi} & W \ar@{->}[r] & 1.}
\end{equation}
A.~Meyer and the first author~\cite{mr} have computed $\ed(N; p)$ in the case, 
where $\Gamma = \on{PGL}_n$, for every prime number $p$.  M.~MacDonald~\cite{macdonald} subsequently
found the exact value of $\ed(N; p)$ for most other split simple groups $\Gamma$. One reason this is of interest is that
\[ \ed(N; p) \geqslant \ed(\Gamma; p); \]
see, e.g.,~\cite[Section 10a]{merkurjev2013essential}.  Let $W_p$ denote 
a Sylow $p$-subgroup of $W$ and $N_p$ denote the preimage of $W_p$ in $N$. 
Then 
\[
\on{ed}(N; p) = \on{ed}(N_p; p);
\]
see~\cite[Lemma 4.1]{mr}. The exact sequence
\[ \xymatrix{ 1 \ar@{->}[r] &  T \ar@{->}[r] & N_p \ar@{->}[r]^{\pi} & W_p \ar@{->}[r] & 1} \]
 is of the form of (\ref{e.torus}) and thus the inequalities~\eqref{e.lmmr} apply to $N_p$.
MacDonald computed the exact value of $\ed(N; p) = \ed(N_p; p)$ for most split simple linear algebraic groups $\Gamma$
by showing that the left hand side and the right hand 
side of the inequalities~\eqref{e.lmmr} for $N_p$ coincide. There are two families of groups $\Gamma$, 
where the exact value of $\ed(N; p)$ remained inaccessible by this method, $\Gamma = \on{SL}_n$ 
and $\Gamma = \on{SO}_{4n}$.~\footnote{The omission of $\on{SL}_n$ from \cite[Remark 5.11]{macdonald} is an oversight;
we are grateful to Mark MacDonald for clarifying this point for us.}
As an application of Theorem~\ref{thm.main}, we will now compute $\ed(N; p)$ in these two remaining cases. 
Our main results are Theorems~\ref{thm.sl} and~\ref{thm.so} below.

\begin{thm} \label{thm.sl} Let $n\geq 1$ be an integer, and let $N$ be the normalizer of a $k$-split maximal torus $T$ in $\on{SL}_{n}$.
Then 

\smallskip
(a) $\ed(N; p) = {n}/{p} + 1$, if $p \geqslant 3$ and $n$ is divisible by $p$, 

\smallskip
(b) $\ed(N; p) = {n}/{2} + 1$, if $p =2$ and $n$ is divisible by $4$, 

\smallskip
(c) $\ed(N; p) = \lfloor {n}/{p} \rfloor$, if $p \geqslant 3$ and $n$ is not divisible by $p$,

\smallskip
(d) $\ed(N; p) = \lfloor {n}/{2} \rfloor$, if $p = 2$ and $n$ is not divisible by $4$.
\end{thm} 

\begin{thm} \label{thm.so}
Let $k$ be a field of characteristic $\neq 2$ and $n\geq 1$ be an integer. Let $N$ be the normalizer of a $k$-split maximal torus of $\on{SO}_{4n}$. Then $\ed_k(N;2)=4n$.
\end{thm}

Our proofs of these theorems will rely on the following simple lemma, which is implicit in~\cite{mr} and~\cite{macdonald}.
Let $F$ be a finite discrete $p$-group, and let $M$ be an $F$-lattice. The symmetric $p$-rank of $M$ is the minimal cardinality $d$ of a finite $H$-invariant $p$-spanning subset $\{ x_1, \ldots, x_d \} \subset M$.  Here ``$p$-spanning" means that the index of the $\mathbb Z$-module spanned by $x_1, ..., x_d$ in $M$ is finite and prime to $p$. Following MacDonald, we will denote the symmetric $p$-rank of $M$
by $\on{SymRank}(M; p)$.

\begin{lemma}\label{symmrank}
Consider an exact sequence $1 \to T \to G \to F \to 1$ of algebraic groups over $k$, as in~\eqref{e.torus}. Assume further that
$T$ is a split torus and $F$ is a constant finite $p$-group. Denote the character lattice of $T$ by $X(T)$, we will view it as an $F$-lattice. Then $\eta(G) \geqslant \on{SymRank}(X(T); p)$.
\end{lemma}

Here $\eta(G)$ denotes the minimal dimension of a $p$-faithful representation of $G$, as defined in the Introduction, 
and $X(T)$ is viewed as an $F$-lattice. If we further assume that the sequence (\ref{e.torus}) in \Cref{symmrank} is split, then, in fact,  
$\eta(G) = \on{SymRank}(X(T); p)$. We shall not need this equality in the sequel, so we leave its proof as 
an exercise for the reader.

\begin{proof}
Let $V$ be a $p$-faithful representation of $G$, of minimal dimension $r=\eta(G)$. As a $T$-representation, $V$ decomposes as the direct sum of characters $\chi_1,\dots,\chi_r$. A simple calculation shows that the $F$-action permutes the $\chi_i$. Let $S\c G$ be the torus generated by the images of the $\chi_i$. By construction, we have an $F$-equivariant homomorphism whose kernel is finite and of order prime to $p$. Passing to character lattices, we obtain an $F$-equivariant homomorphism $X(S)\to X(T)$ whose cokernel is finite and of order prime to $p$. The images of the $\chi_i$ in $X(T)$ form a $p$-spanning subset of $X(T)$ of size $\eta(G)$.
%
\end{proof}

For the proof of Theorem~\ref{thm.sl} we will also need the following lemma.
Let $\Gamma = \on{SL}_n$, $T$ be the diagonal maximal torus, $N$ be the normalizer of $T$ in $\on{SL}_n$, 
$H$ be a subgroup of the Weyl group $W = N/T \simeq \on{S}_n$, and $N'$ be the preimage of
$H$ in $N$. Restricting~\eqref{e.weyl} to $N'$, we obtain an exact sequence
\[\xymatrix{ 1 \ar@{->}[r] &  T \ar@{->}[r] & N' \ar@{->}[r]^{\pi} & H \ar@{->}[r] & 1.}\]

\begin{lemma} \label{lem.SLn}
Let $V_n$ be the natural $n$-dimensional representation of $\on{SL}_n$ and $S$ be the stabilizer in general position
for the restriction of this representation to $N'$. Then (a) $S \cap T = 1$ and (b) $\pi(S) = H \cap \on{A}_n$.
\end{lemma}

Here, as usual, $\on{A}_n$ denotes the alternating group.

\begin{proof} (a) follows from the fact that the $T$-action on $V_n$ is generically free. To prove (b), note that
$\pi(S)$ is the kernel of the action of $H$ on $V_n/T$, where $V_n/T$ is the rational quotient of $V_n$ by the action of $T$;
see, e.g., the proof of~\cite[Proposition 7.2]{lotscher2013essential2}. Consider the dense open subset $\bbG_{\on{m}}^n \subset V_n$ consisting of vectors of the form $(x_1, x_2, \ldots, x_n)$,
where $x_i \neq 0$ for any $i = 1, \ldots, n$. We can identify $\bbG_{\on{m}}^n$ with the diagonal maximal torus in $\on{GL}_n$.
Now 
\[ \xymatrix{ V_n/T \ar@{<-->}[r]^{\simeq \quad} &  (\bbG_{\on{m}})^n/ T \ar@{->}[r]_{\; \; \;\mbox{det}}^{\; \; \; \simeq} & \bbG_{\on{m}}}  \]
where $\on{S}_n$ acts on $\bbG_{\on{m}}$ by $\sigma \cdot t = \on{sign}(\sigma)t$. Thus the kernel of the $H$-action on $V_n/T$ is 
$H \cap \on{A}_n$, as claimed.
\end{proof}

\begin{proof}[Proof of Theorem~\ref{thm.sl}] We will assume that $\Gamma = \on{SL}_n$ and $T$ is the diagonal torus in $\Gamma$.
The inequalities
\begin{equation} \label{e.mark-sln} \lfloor \frac{n}{p} \rfloor \leqslant \ed(N; p) \leqslant \lfloor \frac{n}{p} \rfloor + 1;
\end{equation}
are known for every $n$ and $p$; see \cite[Section 5.4]{macdonald}. We will write $V_n$ for the natural $n$-dimensional
representation of $\on{SL}_n$ (which we will sometimes restrict to $N$ or subgroups of $N$). 

\smallskip
(a) Suppose $n$ is divisible by $p$.
Let $H \simeq (\mathbb Z/p \mathbb Z)^{n/p}$ be the subgroup of $W = N/T \simeq \on{S}_n$ generated by 
the commuting $p$-cycles $(1 \;  2 \; \ldots \; p)$, 
$(p+1 \; p+2 \; \ldots  \; 2p)$, $\ldots$, $(n - p + 1 \; \ldots \; n)$. Since $H$ is a $p$-group, it lies in a
Sylow $p$-subgroup $W_p$ of $\on{S}_n$. Denote the preimage of $H$ in $N$ by $N'$. 
Then $N'$ is a subgroup of $N$ of finite index, so 
\begin{equation} \label{e.N'1}
\ed(N; p) \geqslant \ed(N'; p);
\end{equation}
see~\cite[Lemma 2.2]{brv-annals}.
It thus suffices to show that $\ed(N'; p) = \dfrac{n}{p} + 1$.

\smallskip
{\bf Claim:} $\eta(N') = n$. 

\smallskip
\noindent
Suppose the claim is established. Then $V_n$ is a $p$-faithful representation of $N'$ of minimal dimension. 
Since $p$ is odd, $H$ lies in the alternating group $\on{A}_n$. By Lemma~\ref{lem.SLn}(a), the stabilizer in general position for the $N'$-action on $V$ is isomorphic to $H$.
By \Cref{thm.main}
\[ \ed(N'; p) = \dim(V_n) + \on{rank}(H) - \dim(N') = n + \frac{n}{p} - (n-1) = \frac{n}{p} + 1, \]
and we are done.

To prove the claim, note that $N'$ has a faithful representation $V_n$ of dimension $n$. Hence, $\eta(N') \leqslant n$.
To prove the opposite inequality, $\eta(N') \geqslant n$, it suffices to show that  
\begin{equation} \label{e.symm}
\on{SymRank}(X(T); p) \geqslant n; 
\end{equation}
see~\Cref{symmrank}. Here we view $X(T)$ as an $H$-lattice.
By definition, $\on{SymRank}(X(T); p)$ is the minimal cardinality of a finite H-invariant $p$-spanning subset  $\{ x_1, \ldots, x_d \} \subset X(T)$. The $H$-action on $\{ x_1, \ldots, x_d \}$  gives rise to a permutation representation $\phi \colon H \to \on{S}_d$. 

The permutation representation $\phi$ is necessarily faithful. Indeed, assume the contrary: $1 \neq h$ lies in the kernel of $\phi$. Then $x_1, \ldots, x_d$ lie in $X(T)^h$. On the other hand, it is easy to see that $X(T)^h$ is of infinite index in $X(T)$. Hence,
$\{ x_1, \ldots, x_d \}$ cannot be a $p$-spanning subset of $X(T)$.
This contradiction shows that $\phi$ is faithful. 

Now \cite[Theorem 2.3(b)]{ag} tells us that the order of any abelian $p$-subgroup of $\on{S}_d$ is $\leqslant p^{d/p}$. In particular,
$|H| \leqslant p^{d/p}$. In other words, $p^{n/p} \leqslant p^{d/p}$ or equivalently, $n \leqslant d$. This completes the proof of~\eqref{e.symm} and thus of the claim and of part (a).

\smallskip
(b) When $p = 2$, the argument in part (a) does not work as stated because it is no longer true that $H$ lies in the alternating group $\on{A}_n$. However, when $n$ is divisible by $4$, we can redefine $H$ as follows: 
\[ H = H_1 \times \ldots \times H_{n/4} \hookrightarrow \on{A}_4 \times \ldots \times \on{A}_4 \; \; \text{($n/4$ times)} \; \; \hookrightarrow A_{n}, \]
where $H_i \simeq (\mathbb Z / 2 \mathbb Z)^2$ is the unique normal subgroup of order $4$ in the $i$th copy of $\on{A}_4$. Now $H \simeq \mathbb (Z/ 2 \mathbb Z)^{n/p}$ is a subgroup of $\on{A}_n$, and the rest of the proof of part (a) goes through unchanged.

\smallskip
(c) Write $n = pq + r$, where $1 \leqslant r \leqslant p - 1$. The subgroup of $\on{S}_n$ consisting of permutations
$\sigma$ such that $\sigma(i) = i$ for any $i > pq$, is naturally identified with $\on{S}_{pq}$. Let $P_{pq}$ be a $p$-Sylow subgroup of $\on{S}_{pq}$, and let $N'$ be the preimage of $P_{pq}$ in $N$. Then $[N: N'] = [\on{S}_n : P_{pq}]$ is prime to $p$; hence, it suffices to show that $\ed(N'; p) = \lfloor n/p \rfloor$. In view of \eqref{e.mark-sln}, it is enough to show that $\ed(N';p)\leqslant \lfloor n/p \rfloor$. Since $r\geqslant 1$, as an $N'$-representation, $V_n$ splits as $k^{pq} \oplus k^r$ in the natural way.
Let us now write $k^r$ as $k^{r-1} \oplus k$ and combine $k^{r-1}$ with $k^{pq}$. This yields a decomposition
\[ V_n = k^{n-1} \oplus k \] where the action of $N'$ on $k^{n-1}$ is faithful. 
Now recall that $P_{pq}$ has a faithful $q$-dimensional representation; see, e.g., 
the proof of~\cite[Lemma 4.2]{mr}. Denote this representation by $V'$. Viewing $V'$ as a $q$-dimensional
representation of $N'$ via the natural projection $N' \to P_{pq}$, we obtain a generically free representation
$k^{n-1} \oplus V'$ of $N'$. Thus
\[\ed(N';p)\leqslant \dim(k^{n-1} \oplus V')-\dim(N')=(n-1)+ q-(n-1)= q = \lfloor \frac{n}{p} \rfloor,\]
as desired.

\smallskip
(d) The argument of part (c) is valid for any prime. In particular, if $p = 2$, it proves part (d) in the case, where $n$ is odd.
Thus we may assume without loss of generality that $n \equiv 2 \pmod{4}$.
Let $N'$ be the preimage of $P_n$ in $N$, where $P_n$ is a Sylow $2$-subgroup of $\on{S}_n$. Then the index
$[N: N'] = [\on{S}_n : P_n]$ is finite and odd; hence, $\on{ed}(N; 2) = \on{ed}(N'; 2)$. 
In view of~\eqref{e.mark-sln}, it suffices to show that $\on{ed}(N'; 2) \leqslant n/2$.

Since $n \equiv 2 \pmod 4$, $P_n = P_{n-2} \times P_2$, where $P_2 \simeq \on{S}_2$ is the subgroup of $\on{S}_n$
of order $2$ generated by the $2$-cycle $(n-1, \; n)$. Let $V'$ be a faithful representation
of $P_{n-2}$ of dimension $(n-2)/2$. We may view $V'$ as a representation of $N'$
via the projection $N' \to P_n \to P_{n-2}$.

\smallskip

{\bf Claim:} $V_n \oplus V'$ is a generically free representation of $N'$.

\smallskip
\noindent
If this claim is established, then \[\on{ed}(N') \leqslant \on{dim}(V_n \oplus V') - \on{dim}(N') = 
n + \frac{n-2}{2} - (n-1) = \frac{n}{2},\] and we are done.

To prove the claim, let $S$ the stabilizer in general position for the action of $N'$ on $V_n$.
Denote the natural projection $N' \to P_n$ by $\pi$. By Lemma~\ref{lem.SLn}(a), $S \cap T = 1$. 
In other words, $\pi$ is an isomorphism between $S$ and $\pi(S)$.
Since $P_n = P_{n-2} \times P_2$, the kernel of the $P_n$-action on $V'$ is $P_2$.
It now suffices to show that $S$ acts faithfully on $V'$, i.e., $\pi(S) \cap P_2 = 1$.

By Lemma~\ref{lem.SLn}, $\pi(S) \subset \on{A}_n$, i.e., every permutation in $\pi(S)$ is even.
On the other hand, the non-trivial element of $P_2$, namely the transposition $(n-1, \;  n)$, is odd.
This shows that $\pi(S) \cap P_2 = 1$, as desired.
\end{proof}

\begin{proof}[Proof of Theorem~\ref{thm.so}] By~\cite[Section 5.7]{macdonald}, $\ed(N; 2) \leqslant 4n$. Thus it suffices to show that
$\ed(N; 2) \geqslant 4n$. Let 
\[(\Z/2\Z)_0^{2n}:=\set{(\gamma_1,\gamma_2,\dots,\gamma_{2n})\in (\Z/2\Z)^{2n}:\sum_{i=1}^{2n}\gamma_i=0}.\]
Recall that a split maximal torus $T$ of $\on{SO}_{4n}$ is isomorphic to $(\bbG_{\on{m}})^{2n}$, and  the Weyl group $W$ 
is a semidirect product $A \rtimes \on{S}_{2n}$, where $A$ is an elementary abelian $2$-group 
    $A \simeq (\Z/2\Z)^{2n-1}$. Here $A$ is the multiplicative group of $2n$-tuples 
    ${\bf \epsilon} = (\epsilon_1, \ldots, \epsilon_{2n})$, where each $\epsilon_i$ is $\pm 1$, and $\epsilon_1 \epsilon_2 \ldots \epsilon_{2n} = 1$. 
    $\on{S}_{2n}$ acts on $A$ by permuting $\epsilon_1, \dots, \epsilon_{2n}$. The action of $W$ on $(t_1, \dots, t_{2n}) \in T$ is
    as follows: $\on{S}_{2n}$ permutes $t-1, \ldots, t_{2n}$, and ${\bf \epsilon}$ takes each $t_i$ to $t_i^{\epsilon_i}$.
    
    Let $H$ be the subgroup of $W$ generated by elements $(\epsilon_1, \ldots, \epsilon_{2n}) \in A$, with $\epsilon_1 = \epsilon_2$, 
    $\epsilon_3 = \epsilon_4$, $\ldots$, $\epsilon_{2n-1} = e_{2n}$, 
    and the $n$ disjoint 2-cycles $(1, 2), (3, 4), \ldots, (2n-1, 2n)$ in $\on{S}_{2n}$. It is easy to see that
    these generators are of order $2$ and commute with each other, so that
    $H \simeq (\Z/2\Z)^{n}$. Let $N'$ be the preimage of $H$ in $N$. 
    
Note that $H$ arises as a stabilizer in general position of the natural $4n$-representation $V_{4n}$ of $N$ (restricted 
from $\on{SO}_{4n}$).
Here $(t_1, \dots, t_{2n}) \in T$ acts on $(x_1, \ldots, x_{2n}, y_1, \ldots, y_{2n}) \in V_{4n}$ by $x_i \mapsto t_i x_i$
and $y_i \mapsto t_i^{-1} y_i$ for each $i$. The symmetric group $\on{S}_{2n}$ simultaneously permutes $x_1, \ldots, x_{2n}$ and
$y_1, \ldots, y_{2n}$; ${\bf \epsilon} \in A$ leaves $x_i$ and $y_i$ invariant if $\epsilon_i = 1$ and switches them 
if $\epsilon_i = -1$.

Note that $N'$ is a subgroup of finite index in $N$. Hence, $\ed(N; 2) \geqslant \ed(N'; 2)$, and it suffices to show that
$\ed(N'; 2) \geqslant 4n$.

\smallskip
{\bf Claim:} $\eta(N') = 4n$.

\smallskip
Suppose for a moment that the claim is established. 
Then $V_{4n}$ is a $2$-faithful representation of $N'$ of minimal dimension. As we mentioned above, 
a stabilizer in general position for this representation is isomorphic to $H$.
By \Cref{thm.main},
\[ \ed(N'; 2) = \dim(V_{4n}) + \on{rank}(H) - \dim(N') =  4n + 2n - 2n = 4n ,\]
and we are done.

To prove the claim, note that $\eta(N') \leqslant 4n$, since $N'$ has a faithful representation $V_{4n}$ of dimension $4n$.
By Lemma~\ref{lem.SLn}, in order to establish the opposite inequality, $\eta(N') \geqslant 4n$, it suffices to show
that $\on{SymRank}(X(T); 2) \geqslant 4n$. To prove this last inequality, we will use the same argument as in the proof of Theorem~\ref{thm.sl}(a).
Recall that $\on{SymRank}(X(T); 2)$ is the minimal size of an $H$-invariant $2$-generating set $x_1, \dots, x_d$ of $X(T)$.
The $H$-action on $x_1, \ldots, x_d$ induces a permutation 
representation $\phi \colon H \to \on{S}_d$. Once again, this representation has to be faithful. By \cite[Theorem 2.3(b)]{ag},
$|H| \leqslant 2^{d/2}$. In other words, $2^{2n} \leqslant 2^{d/2}$, or equivalently, $d \geqslant 4n$, as claimed.
\end{proof}

\section*{Acknowledgements} We are grateful to Jerome Lefebvre, Mathieu Huruguen and Mikhail Borovoi for stimulating discussions
related to Conjecture~\ref{conj.lmmr}.

\end{document}